\documentclass{article}
\usepackage[utf8]{inputenc}
\usepackage{amsmath}
\usepackage{amsfonts} 
\usepackage{amsthm}
\usepackage{verbatim}
\usepackage{etoolbox}
\usepackage{dirtytalk}
\usepackage{tikz-cd}
\usepackage{derivative}
\usepackage{amssymb}
\usepackage{extarrows}

\theoremstyle{plain}
\newtheorem{thm}{Theorem}[section]
\newtheorem{lem}[thm]{Lemma}

\newtheorem{prop}[thm]{Proposition}

\newtheorem*{claim*}{Claim}

\theoremstyle{definition}

\theoremstyle{remark}

\AtBeginEnvironment{thm}{\setcounter{equation}{0}}
\AtBeginEnvironment{lem}{\setcounter{equation}{0}}
\AtBeginEnvironment{obs}{\setcounter{equation}{0}}
\AtBeginEnvironment{defi}{\setcounter{equation}{0}}
\AtBeginEnvironment{prop}{\setcounter{equation}{0}}
\AtBeginEnvironment{cor}{\setcounter{equation}{0}}

\AtBeginEnvironment{thm}{\setcounter{case}{0}}
\AtBeginEnvironment{lem}{\setcounter{case}{0}}
\AtBeginEnvironment{obs}{\setcounter{case}{0}}
\AtBeginEnvironment{defi}{\setcounter{case}{0}}
\AtBeginEnvironment{prop}{\setcounter{case}{0}}
\AtBeginEnvironment{cor}{\setcounter{case}{0}}

\AtBeginEnvironment{thm}{\setcounter{claim}{0}}
\AtBeginEnvironment{lem}{\setcounter{claim}{0}}
\AtBeginEnvironment{obs}{\setcounter{claim}{0}}
\AtBeginEnvironment{defi}{\setcounter{claim}{0}}
\AtBeginEnvironment{prop}{\setcounter{claim}{0}}
\AtBeginEnvironment{cor}{\setcounter{claim}{0}}

\newcommand{\R}{\mathbb{R}}

\newcommand{\RR}{\mathcal{R}}
\newcommand{\PP}{\mathcal{P}}

\newcommand{\I}{\mathcal{I}}
\newcommand{\Z}{\mathcal{Z}}

\newcommand{\indet}{\textnormal{indet}}
\newcommand{\dom}{\textnormal{dom}}

\title{Extension of $k$-regulous functions from varieties of arbitrary dimension}
\author{Juliusz Banecki}
\date{}
\AtEndDocument{%
  \par
  \textsc{Juliusz Banecki,
Faculty of Mathematics and Computer Science,
Jagiellonian University,
ul. Lojasiewicza 6, 30-348 Krakow, Poland}\\
\textit{E-mail address}: 
\texttt{juliusz.banecki@student.uj.edu.pl}
}

\begin{document}
\maketitle

\begin{abstract}
We prove that a $k$-regulous function defined on a non-singular affine variety can always be extended to the entire affine space.
\end{abstract}

\section{Introduction}
In the paper a real affine variety is a Zariski closed subset $X$ of $\R^n$ for some $n$. By $\PP(X)$ we denote the ring of polynomials on $X$. By $\PP(X)_{x_0}$ we denote the localisation of $\PP(X)$ at the maximal ideal of functions vanishing at $x_0$, which can be interpreted as the ring of germs of regular functions at $x_0$. The variety is said to be non-singular at $x_0$ if the corresponding local ring is regular. The definitions given here are compatible with the ones given in \cite{RAG,Mangolte}.

Given a non-singular affine variety $X\subset \R^n$ and a non-empty Euclidean open set $U\subset X$, we say that a function $f:U\rightarrow \R$ is $k$-regulous if it is of class $\mathcal{C}^k$ and there exist two polynomials $P,Q\in\PP(X)$, such that $Q$ does not vanish identically on any irreducible component of $X$ and $P=fQ$ on $U$. Clearly, the polynomials uniquely determine $f$, so in particular if $X$ is irreducible then $\RR^k(U)$ naturally embeds in the field of rational functions on $X$.

This definition is not completely standard, as for $X\neq \R^n$ some authors consider only these functions which admit $k$-regulous extensions to the entire space $\R^n$. From now on however, this distinction is no longer important, as the goal of the paper is to prove the following theorem:
\begin{thm}\label{main_thm}
Let $X\subset\R^n$ be a non-singular affine variety and let $f\in\RR^k(X)$. Then, there exists $F\in\RR^k(\R^n)$ such that $F\vert_X=f$.
\end{thm}

Earlier, the result was known to hold in two particular cases:
\begin{enumerate}
    \item when $k=0$; due to \cite{Kollar-Nowak}. In the paper Koll\'ar and Nowak have also proven the converse, saying that for every $k$, the restriction of a $k$-regulous function on $\R^n$ to $X$ is $k$-regulous.
    \item when $\dim(X)=2$; due to \cite{two-dimensional}. The approach presented by the author there is rather technical and makes significant use of the resolution of singularities.
\end{enumerate}

The two above papers rely on vaguely related ideas, which do not seem sufficient to tackle the general case. The problem has remained open from the introduction of $k$-regulous functions in \cite{francuska} up until now. The approach given in the current paper is quite short as compared to \cite{two-dimensional}, and it is based on some innovative ideas from the recent paper \cite{rr=>urr} regarding retract rational varieties. Even though the results presented there seem to have little to do with $k$-regulous functions, it turns out that the approach given there is just what we need in the current problem.
\section{Proofs}
We need the following proposition playing an important role in \cite{rr=>urr}:
\begin{prop}[{\cite[Proposition 2.2]{rr=>urr}}]\label{alg_prop}
Let $X\subset \R^n$ be a non-singular irreducible affine variety of dimension $m$, and let $x_0\in X$. Let $I\subset \PP(X)$ be a non-zero ideal such that $I\subset \mathfrak{m}_{x_0}$. Then, there exists a polynomial mapping $\sigma:X\times \R^{n-m}\rightarrow \R^n$, such that
\begin{enumerate}
    \item $\sigma(x,0)=x$ for $x\in X$,
    \item the derivative of $\sigma$ at $(x_0,0)$ is a isomorphism,
    \item the induced homomorphism
    \begin{equation*}
        \PP(\R^n)_{x_0}\rightarrow (\PP(X)[t_1,\dots,t_{n-m}]/I(t_1,\dots,t_{n-m}))_{\mathfrak{m}_{(x_0,0)}}
    \end{equation*}
    is surjective.
\end{enumerate}
\end{prop}

Furthermore, we need an analogue of \cite[Lemma 3.1]{rr=>urr}:
\begin{lem}\label{denom_reduction_reg}
Let $P,Q\in \PP(\R^n)$ two polynomials. Let $X$ be an irreducible non-singular variety, and let $\pi:X\rightarrow \R^n$ be a regular mapping, such that $Q\circ \pi$ is not identically equal to zero and $\frac{P}{Q}\circ \pi\in \RR^k(X)$. Let $x_0$ be a point of $X$ and let $\varphi:(X,x_0)\rightarrow \R^n$ be a germ of a regular mapping, such that for $1\leq i \leq n$ its $i$-th coordinated $\varphi_i$ satisfies
\begin{equation*}
    \varphi_i-\pi_i\in (Q\circ\pi)\mathfrak{m}_{x_0}\text{ in the ring }\PP(X)_{x_0}.
\end{equation*}
Then, there exists a Zariski neighbourhood $V$ of $x_0$ such that 
\begin{enumerate}
    \item $V\subset \dom (\varphi)$,
    \item for $x\in V$ we have $Q\circ \pi(x)=0\iff Q\circ \varphi(x)=0$,
    \item $\frac{P}{Q}\circ \varphi\in\RR^k(V)$.
\end{enumerate}
\begin{proof}
Note that for $v,w\in\R^n$, the polynomials $P$ and $Q$ can be written as
\begin{align*}
    P(v+w)=P(v)+\sum_i w_i P_i(v,w),\\
    Q(v+w)=Q(v)+\sum_i w_i Q_i(v,w),
\end{align*}
for some $P_i,Q_i\in \PP(\R^n)$. Take $V$ as the locus of points $x\in X$ satisfying the following conditions:
\begin{enumerate}
    \item $\varphi$ is defined at $x$,
    \item $\frac{\varphi_i-\pi_i}{Q\circ \pi }$ is regular at $x$ for all $i$,
    \item $1+\sum_i \frac{\varphi_i(x)-\pi_i(x)}{Q\circ\pi(x)} Q_i(\pi(x),\varphi(x)-\pi(x))\neq 0$.
\end{enumerate}
Then
\begin{equation*}
    Q\circ\varphi(x)=Q\circ\pi(x)\left(1+\sum_i \frac{\varphi_i(x)-\pi_i(x)}{Q\circ\pi(x)} Q_i\big(\pi(x),\varphi(x)-\pi(x)\big)\right)
\end{equation*}
and so the second point follows. Furthermore
\begin{equation*}
    \frac{P}{Q}\circ \varphi(x)=\frac{\frac{P}{Q}\circ \pi(x)+\sum_i \frac{\varphi_i(x)-\pi_i(x)}{Q\circ \pi(x)} P_i\big(\pi(x),\varphi(x)-\pi(x)\big)}{1+\sum_i \frac{\varphi_i(x)-\pi_i(x)}{Q\circ\pi(x)} Q_i\big(\pi(x),\varphi(x)-\pi(x)\big)}\in\RR^k(V).
\end{equation*}
\end{proof}
\end{lem} 

Recall that a set $U\subset \R^n$ is \emph{constructibly open}, if it is open in the Euclidean topology, and it can be written as a finite union of Zariski locally closed subsets of $\R^n$. With the machinery at hand, we can prove the following proposition, from which Theorem \ref{main_thm} follows easily:
\begin{prop}\label{main_prop}
Let $X\subset \R^n$ be a non-singular irreducible variety and let $f\in\RR^k(X)$. Fix a point $x_0\in X$. Then, there exists a constructibly open neighbourhood $U$ of $x_0$ in $\R^n$ and a regulous function $F\in\RR^k(U)$ such that $F\vert_{X\cap U}=f\vert_{X\cap U}$.

\begin{proof}
Set $m:=\dim (X)$. Write $f$ as $f=\frac{p}{q}$ with $p,q\in\PP(X)$ and $q\neq 0$, and choose any $P,Q\in\PP(\R^n)$ such that $p=P+\I(X),q=Q+\I(X)$. Without loss of generality $q(x_0)=0$, for otherwise we can take $F:=\frac{P}{Q}$. Applying Proposition \ref{alg_prop} to the ideal $I=(q)$, we find a polynomial mapping $\sigma:X\times\R^{n-m}\rightarrow \R^n$ such that
\begin{enumerate}
    \item $\sigma(x,0)=x$ for $x\in X$,
    \item the derivative of $\sigma$ at $(x_0,0)$ is a isomorphism,
    \item the induced homomorphism
    \begin{equation*}
        \PP(\R^n)_{x_0}\rightarrow (\PP(X\times \R^{n-m})/(q)(t_1,\dots,t_{n-m}))_{\mathfrak{m}_{(x_0,0)}}
    \end{equation*}
    is surjective.
\end{enumerate}

Set $Y:=\Z(q)\times\R^{n-m}\subset X\times \R^{n-m}$. The composed homomorphism
\begin{equation*}
    \PP(\R^n)_{x_0}\rightarrow\left(\PP(X\times \R^{n-m})/(q)(t_1,\dots,t_{n-m})\right)_{\mathfrak{m}_{(x_0,0)}}\rightarrow\PP(Y)_{(x_0,0)}
\end{equation*}
is then surjective as well. It follows that, if we define $Z:=\overline{\sigma(Y)}^{Zar}$, then $\sigma$ induces an isomorphism of germs of affine varieties
\begin{equation}\label{eq_local_iso}
    (Y,(x_0,0))\xlongrightarrow{\cong} (Z,x_0).
\end{equation}

For $1\leq i\leq n$ let $\psi_i\in\PP(\R^n)_{x_0}$ be a regular germ such that
\begin{equation*}
    \psi_i\circ\sigma-x_i\in (q)(t_1,\dots,t_{n-m}).
\end{equation*}
These together give a regular germ $\psi:(\R^n,x_0)\rightarrow \R^n$. Set $F:=\frac{P}{Q}\circ \psi$.

By Lemma \ref{denom_reduction_reg} applied to the mappings 
\begin{align*}
    \pi:X\times\R^{n-m}\rightarrow \R^n\\
    \pi(x,t):=x\in X\subset \R^n,\\
    \varphi:(X\times\R^{n-m},(x_0,0))\rightarrow \R^n\\
    \varphi(x,t):=\psi\circ\sigma(x,t)
\end{align*}
there exists a Zariski neighbourhood $V$ of $(x_0,0)$ in $X\times\R^{n-m}$ such that $F\circ \sigma \in\RR^k(V)$ and for $(x,t)\in V$ we have
\begin{equation}\label{eq_sets}
    Q\circ\psi\circ \sigma(x,t)=0\iff (x,t)\in Y.
\end{equation}
Decreasing $V$ if necessary we can assume that the derivative of $\sigma$ is non-singular on $V$. Furthermore, we can assume that $\sigma(V)\subset \dom (\psi)$. Lastly, thanks to (\ref{eq_local_iso}), we can assume that $\sigma(V\cap Y)$ is Zariski open in $Z$, so in particular it is constructible in $\R^n$.

Define
\begin{align*}
    W&:=\indet (\psi) \cup \{x\in \dom(\psi):Q\circ\psi(x)=0\},\\
    U&:=\R^n\backslash W\cup \sigma(V).
\end{align*}
As $\sigma$ is a local diffeomorphism between $V$ and $\sigma(V)$, it follows that $F\in\RR^k(\sigma(V))$. Also $F\in\RR(\R^n\backslash W)$, so $F\in\RR^k(U)$. Furthermore, for $x\in X$ sufficiently close to $x_0$ by assumption $\psi(x)=\psi\circ \sigma(x,0)=\pi(x,0)=x$, so $F\vert_X=f$ as a rational function. It hence suffices to check that $U$ is constructible in $\R^n$.

We first claim that 
\begin{equation*}
    \sigma(V)\cap W= \sigma(V\cap Y).
\end{equation*}
Indeed, let $(x,t)\in V$. Then, $\psi$ by assumption is regular at $\sigma(x,t)$, so thanks to (\ref{eq_sets}):
\begin{equation*}
    \sigma(x,t)\in W \iff Q\circ \psi \circ \sigma(x,t)=0\iff (x,t)\in Y.
\end{equation*}

Now we have that
\begin{equation*}
    U=\R^n\backslash W\cup \sigma(V)=\R^n\backslash W \cup (\sigma(V)\cap W)=\R^n\backslash W \cup \sigma(V\cap Y)
\end{equation*}
is a constructible set, and the conclusion follows.
\end{proof}
\end{prop}

\begin{proof}[Proof of Theorem \ref{main_thm}]
Let $X$ and $f\in\RR^k(X)$ be as in the theorem. Consider the ideal 
\begin{equation*}
    I:=\{Q\in\RR^k(\R^n):\exists_{P\in\RR^k(\R^n)}P\vert_X=fQ\vert_X\}.
\end{equation*}
It suffices to prove that $1\in I$. 

First of all, clearly $I$ contains the ideal of functions vanishing on $X$, so $\Z(I)\subset X$. Fix a point $x_0\in X$ and let $X_0$ be the irreducible component of $X$ containing $x_0$. Applying Proposition \ref{main_prop}, we get that there exists a constructibly open neighbourhood $U$ of $x_0$ in $\R^n$ and a $k$-regulous function $F\in\RR^k(U)$ such that
\begin{equation*}
    F\vert_{U\cap X_0}=f\vert_{U\cap X_0}.
\end{equation*}
According to \cite[Th\'{e}or\`{e}me 6.4 and Corollaire 4.4]{francuska}, since $U$ is constructibly open, there exists $Q\in\RR^k(\R^n)$ with $\Z(Q)=\R^n\backslash U$. Then, thanks to \cite[Lemme 5.1]{francuska}, $F$ can be written as $F=\frac{P}{Q^N}$, where $P\in\RR^k(\R^n)$ and $N$ is a sufficiently large integer. Take $Q_1\in\RR(\R^n)$ to be any regular function vanishing precisely on $X\backslash X_0$. Then, as $U\cap X_0$ is Euclidean dense in $X_0$, we have
\begin{equation*}
    PQ_1\vert_X=fQ^NQ_1\vert_X,
\end{equation*}
so $Q^NQ_1\in I$. Hence $x_0\not \in \Z(I)$. As $x_0$ was an arbitrary point of $X$, we get $\Z(I)=\emptyset$. Thanks to the weak regulous Nullstellensatz (\cite[Proposition 5.23]{francuska}), we get $1\in I$. The conclusion follows.
\end{proof}
\bibliographystyle{plain}
\bibliography{references}

\begin{thebibliography}{1}

\bibitem{two-dimensional}
Juliusz Banecki.
\newblock Extensions of k-regulous functions from two-dimensional varieties.
\newblock {\em Mathematische Annalen}, Sep 2024.

\bibitem{rr=>urr}
Juliusz Banecki.
\newblock Retract rational varieties are uniformly retract rational.
\newblock {\em arXiv preprint}, 2411.17892, 2024.

\bibitem{RAG}
Jacek Bochnak, Michel Coste, and Marie-Fran{\c{c}}oise Roy.
\newblock {\em Real Algebraic Geometry}.
\newblock Springer Berlin Heidelberg, Berlin, Heidelberg, 1998.

\bibitem{francuska}
Goulwen Fichou, Johannes Huisman, Frédéric Mangolte, and Jean-Philippe Monnier.
\newblock Fonctions régulues.
\newblock {\em Journal für die reine und angewandte Mathematik}, 2016(718):103--151, 2016.

\bibitem{Kollar-Nowak}
J{\'a}nos Koll{\'a}r and Krzysztof Nowak.
\newblock Continuous rational functions on real and p-adic varieties.
\newblock {\em Mathematische Zeitschrift}, 279(1):85--97, Feb 2015.

\bibitem{Mangolte}
Fr\'ed\'eric Mangolte.
\newblock {\em Real algebraic varieties}.
\newblock Springer Monographs in Mathematics, 2020.

\end{thebibliography}
\end{document}